\documentclass[a4paper,11pt]{amsart}
\usepackage{amsmath}
\usepackage{amsthm,amsfonts,amssymb,graphics}
\usepackage{pgfplots}
\usepackage[foot]{amsaddr}

\usepackage{hyperref}
\usepackage[normalem]{ulem}
\usepackage{enumerate}
\usepackage{geometry}
\usepackage{dsfont}

\newtheorem{theorem}{Theorem}[section]

\newtheorem{proposition}[theorem]{Proposition}

\newtheorem{assumption}[theorem]{Assumption}

\theoremstyle{remark}
\newtheorem{remark}[theorem]{Remark}

\numberwithin{equation}{section}

\newcommand \id{\mathds 1}
\newcommand {\R} {\mathbb{R}}

\DeclareMathOperator*{\argmax}{arg\,max}

\begin{document}
\title[Boundedness of the nodal domains of additive Gaussian fields]{Boundedness of the nodal domains of \\ additive Gaussian fields}
\author{Stephen Muirhead}
\email{smui@unimelb.edu.au}
\address{School of Mathematics and Statistics, University of Melbourne}
\begin{abstract}
We study the connectivity of the excursion sets of additive Gaussian fields, i.e.\ stationary centred Gaussian fields whose covariance function decomposes into a sum of terms that depend separately on the coordinates. Our main result is that, under mild smoothness and correlation decay assumptions, the excursion sets $\{f \le \ell\}$ of additive planar Gaussian fields are bounded almost surely at the critical level $\ell_c = 0$. Since we do not assume positive correlations, this provides the first examples of continuous non-positively-correlated stationary planar Gaussian fields for which the boundedness of the nodal domains has been confirmed. By contrast, in dimension $d \ge 3$ the excursion sets have unbounded components at all levels.
\end{abstract}
\date{\today}
\thanks{}
\thanks{Supported by the Australian Research Council (ARC) Discovery Early Career Researcher Award DE200101467.}

\maketitle

\textit{This is a minor correction to the published version of this article (Theor.\ Probab.\ Math.\ Stat.\ 106, 143–155, 2022). The only difference is that the display
 \[  \mathbb{P}(I  \in B_i^+ )^{\textrm{tahn}(s/2)}   \mathbb{E} \big[  e^{2\theta S }    \big]    \]
   on the page 12 of this version has been corrected to 
   \[    \mathbb{P}(I  \in B_i^+ )^{\textrm{tahn}(s/2)}   \mathbb{E} \big[  e^{2\theta S } \id_{I  \in B_i^+  }   \big]   .  \]}
   
\section{Introduction}

Let $f$ be a continuous centred stationary Gaussian field on $\mathbb{R}^2$. The question of whether the excursion sets $\{f \le \ell\}$, $\ell \in \mathbb{R}$, contain an unbounded connected component has been of interest to mathematicians and physicists for many decades \cite{dyk70,is92,zs71}. Molchanov and Stepanov \cite{ms83a,ms83b} gave general conditions under which the percolation threshold is non-trivial, i.e.\ there exists an $\ell_c \in \mathbb{R}$ such that:
\begin{itemize}
\item If $\ell < \ell_c$, almost surely all the connected components of $\{f \le \ell\}$ are bounded;
\item If $\ell > \ell_c$, almost surely $\{f \le \ell\}$ has an unbounded connected component.
\end{itemize}
More recently it has been shown that $\ell_c = 0$ under very mild conditions (\cite{mrv20}, and see also \cite{rv20,mv20,riv21,gv21} for quantitative versions of this result under stronger conditions). The question of the absence of percolation at criticality (i.e.\ whether the nodal domains $\{f \le 0\}$, or equivalently the nodal set $\{f = 0\}$, have only bounded connected components) appears to be more challenging, and thus far has only been established for positively-correlated fields~\cite{al96} (for Gaussian fields on the lattice $\mathbb{Z}^2$, it has also been shown for a \textit{perturbative} class of non-positively-correlated fields~\cite{bg17b}). Recall that the analogous question is still open in general for Bernoulli percolation in dimension $d \ge 3$, although the planar case was settled long ago~\cite{ha60}.

\smallskip
In this paper we consider this question for the class of \textit{additive Gaussian fields}. These are centred stationary Gaussian fields $f$ whose covariance kernel $\kappa(x) = \mathbb{E}[f(0) f(x)]$ can be written as
\[ \kappa(x) = \kappa(x_1, x_2) = K_1(x_1) + K_2(x_2)  \]
where $x = (x_1, x_2) \in \mathbb{R}^2$. A simplifying feature of these fields is that, equivalently, they are Gaussian fields which can be decomposed as
\begin{equation}
\label{e:g}
 f(x) = f(x_1, x_2) =  g_1(x_1) + g_2(x_2) , 
 \end{equation}
where $g_i$ are independent centred stationary Gaussian processes with respective covariance kernel $K_i$. See \cite{dgr12} for a discussion of additive Gaussian fields in the context of statistical modelling, and Figure \ref{f:additive} for an illustration of their nodal domains.

\subsection{Boundedness of the nodal domains}
We shall make the following assumptions:

\begin{assumption}
\label{a}
For each $i = 1,2$, $g_i$ is not identically zero and satisfies: 
\begin{enumerate}
\item (Smoothness) $g_i$ is almost surely $C^2$-smooth;
\item (Decay of correlations) As $|x| \to \infty$, $K_i(x)  \log  |x| \to 0$.
\end{enumerate}
\end{assumption}

We use the $C^2$-smoothness and the `Breman condition' $K_i(x)  \log  |x| \to 0$ in order to access extreme value theory for smooth Gaussian processes (see, e.g., \cite{llr83} as well as Propositions \ref{p:scaling}--\ref{p:ppc} below), although we believe that the results should be true under weaker assumptions. 

\smallskip
Our main result is that, under the above assumptions, the critical level is $\ell_c = 0$ and the nodal domains are bounded:

\begin{theorem}
\label{t:main}
Under the above assumptions:
\begin{itemize}
\item If $\ell \le 0$, almost surely all the connected components of $\{f \le \ell\}$ are bounded;
\item If $\ell > 0$, almost surely $\{f \le \ell\}$ has a unique unbounded connected component.
\end{itemize}
In particular, for all $\ell \in \mathbb{R}$ the level set $\{f = \ell\}$ has bounded components almost surely.
\end{theorem}

\begin{remark}
In Theorem \ref{t:main} we do \textit{not} assume that $f$ is positively-correlated, and so this provides the first examples of smooth non-positively-correlated Gaussian fields for which the boundedness of the nodal domains has been established rigorously.
\end{remark}

\begin{remark}
\label{r:degen}
While the main novelty of Theorem \ref{t:main} is the boundedness of the nodal domains, even the fact that $\ell_c = 0$ is new for these fields. This is since additive Gaussian fields have long-range correlations, due to the degeneracy in the covariance structure, and so fail to satisfy the assumptions imposed in other works that prove $\ell_c = 0$ (e.g.\ \cite{mrv20}). Indeed, the covariance kernel $\kappa$ does not decay in all directions since $\kappa(x_1,0) = K_1(x_1) + K_2(0) \to K_2(0) > 0$ as $x_1 \to \infty$, although correlations do decay along any ray that is not parallel to the coordinate axes.
\end{remark}

\begin{figure}[h]
\centering
\includegraphics[scale=0.42]{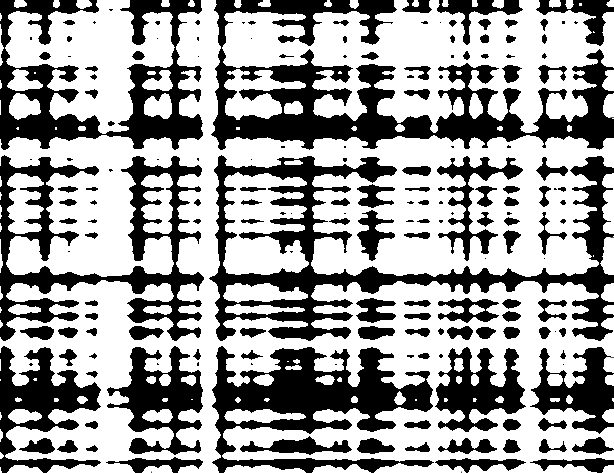}
\hspace{0.2cm}
\includegraphics[scale=0.487]{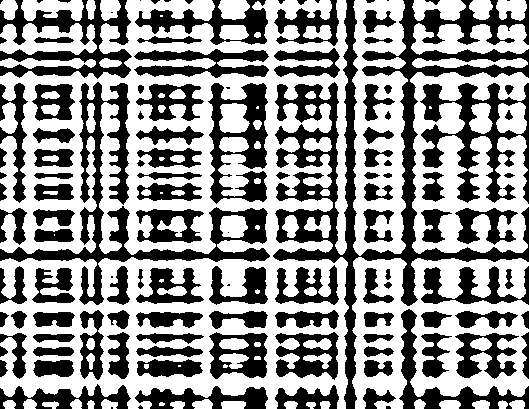}
\caption{The nodal domains $\{f \le 0\}$ (in black) for the additive Gaussian fields with $K_1(x) = K_2(x) = e^{-x^2}$ (left frame) and $K_1(x) = K_2(x) = \cos(x) e^{-x^2}$ (right frame). Theorem~\ref{t:main} states that these sets have bounded connected components almost surely. Note that the first field is positively-correlated while the second field is not.}
\label{f:additive}
\end{figure}

\newpage

In contrast to the planar case, for additive Gaussian fields in dimension $d \ge 3$ the nodal domains $\{f \le 0\}$ \textit{do} contain an unbounded connected component. Indeed if $d \ge 3$ the critical level is $\ell_c = -\infty$, and \textit{every} excursion set $\{f \le \ell\}$ contains an unbounded connected component:

\begin{theorem}
\label{t:3d}
Let $d \ge 3$ and let $f : \mathbb{R}^d \to \mathbb{R}$ be a Gaussian field such that
\[  f(x) = f(x_1, \ldots , x_d) =  g_1(x_1) + \ldots + g_d(x_d) ,   \]
where $g_i$ are independent centred stationary Gaussian processes satisfying Assumption~\ref{a}. Then  $\ell_c = - \infty$, i.e.\ for every $\ell \in \mathbb{R}$, $\{ f \le \ell\}$ has an unbounded connected component almost surely.
\end{theorem}

\begin{remark}
The conclusion of Theorem \ref{t:3d} is markedly different to the expected behaviour for generic stationary Gaussian fields in dimension $d \ge 3$, for which it is believed, and in some cases known \cite{drrv21}, that $\ell_c \in (-\infty, 0)$.
 \end{remark}

\begin{remark}
Unlike in the planar case, the proof of Theorem \ref{t:3d} does not settle whether the unbounded component of $\{f \le \ell\}$ is unique, although it is natural to expect this.
\end{remark}

\subsection{The critical phase} Returning to the planar case, we next turn our attention to features of the critical phase. First we ask whether, as for other planar percolation models (see, e.g., \cite{gri99}), `box-crossing estimates' hold at criticality, i.e.\ whether for every $\rho > 0$,
 \begin{equation}
\label{e:rsw}
0 <  \liminf_{R \to \infty} \mathbb{P}(\textrm{Cross}_0(R, \rho R) ) \le  \limsup_{R \to \infty} \mathbb{P}(\textrm{Cross}_0(R, \rho R) ) < 1 ,
\end{equation}
where $\textrm{Cross}_\ell(a,b)$ is the event that there is a left-right path in $\{f \le \ell\} \cap ( [0,a] \times [0,b])$, i.e.\ a path that intersects $\{0\} \times [0, b]$ and $\{a\} \times [0,b]$. 

\begin{theorem}
\label{t:rsw}
The box-crossing estimates \eqref{e:rsw} hold if and only if  $K_1(0) = K_2(0)$. If $K_1(0) \neq K_2(0)$ then instead, for every $\rho > 0$, as $R \to \infty$,
\[ \mathbb{P}(\textrm{Cross}_0(R, \rho R) )  \to  \begin{cases} 1 & \text{if } K_1(0) < K_2(0) ,\\ 0 & \text{if } K_2(0) < K_1(0) .\end{cases} \]
\end{theorem}

While the behaviour in the case $K_1(0) \neq K_2(0)$ may be surprising when compared with other percolation models, it is best understood as an artefact of the lack of symmetry under rotation by $\pi/2$. Indeed, in general one has the analogue of \eqref{e:rsw} for the events $\textrm{Cross}_0(R, \rho R^{K_1(0)/K_2(0)})$, i.e.\ box-crossing estimates hold after appropriate rescaling.

\smallskip
Finally we consider the size of the `critical window', i.e.\ the neighbourhood around $\ell_c = 0$ for which the connectivity of the excursion sets $\{f \le \ell\}$ behaves `critically' at scale $R$. Restricting to the case $K_1(0) = K_2(0)$, we prove that the size of this window is of order $1/\sqrt{\log R}$.

\begin{theorem}
\label{t:cw}
Suppose $K_1(0) = K_2(0)$ and let $\ell_R \ge 0$. If $\ell_R \sqrt{\log R} \to \infty$ as $R \to \infty$ then
\[   \mathbb{P}( \textrm{Cross}_{\ell_R}(R, R) )  \to 1 , \]
whereas if $\limsup_{R \to \infty} \ell_R \sqrt{\log R} < \infty$ then
\[  \limsup_{R \to \infty}  \mathbb{P}( \textrm{Cross}_{\ell_R}(R, R) )  < 1 . \]
\end{theorem}

Let us highlight two aspects of this result. First, it presents a rare instance among percolation models in which the size of the critical window can be computed precisely. Second, $1/\sqrt{\log R}$ is much larger than the corresponding (conjectured) size of the critical window for models in the Bernoulli universality class, expected to be of order~$R^{-1/\nu + o(1)}$ for $\nu = 4/3$ the \textit{correlation length exponent} (see, e.g., \cite{gri99}). This is quite natural since, as we shall see in Section \ref{s:2}, the degeneracies in the model mean that crossing events $\textrm{Cross}_0(R, R)$ can be well-approximated by exceedence events for the extrema of the Gaussian processes $g_i$ defined in \eqref{e:g}, and $1/\sqrt{\log R}$ is the scale of the fluctuations of the extrema at scale $R$.

\subsection{Further discussion and related models}
While here we study stationary Gaussian fields, one can generalise the set-up by considering $f(x) = f(x_1, x_2) = g_1(x_1) + g_2(x_2)$ where $g_i$ are independent processes not required to be stationary. In the case that $g_i$ are Brownian motions, this model is known as \textit{additive Brownian motion}, and the boundedness of its nodal domains was proven in \cite{dm01} (see also \cite{pet08} for a discrete counterpart of this result for \textit{corner percolation}, which can be viewed as the collection of level sets of additive random walks). 

\smallskip
One important difference to the stationary case is that additive Brownian motion does \textit{not} exhibit a phase transition at the critical level $\ell_c = 0$, and in fact the excursion sets $\{f \le \ell\}$ are almost surely bounded for \textit{every} $\ell \in \mathbb{R}$ (i.e.\ $\ell_c = \infty$). On the other hand, the structural properties of Brownian motion (independent increments etc.) allow one to go further in describing the geometry of the nodal sets, for instance computing the critical exponents that govern the diameter and boundary length (see, e.g., \cite{pet08} for critical exponents in the discrete set-up). Notably these exponents are different from those of Bernoulli percolation, due to the long-range dependence in the model, and we conjecture this is also true in the stationary case.

\smallskip
One can generalise the model in a different direction by considering the field $f(x) = f(x_1, x_2) = \sum_{i = 1,\ldots,k} g_i( \langle x, v_i \rangle)$ where $g_i$ are independent stationary Gaussian processes and $v_i \in \mathbb{S}^1$ are distinct directions on the unit sphere that are not co-linear; the model we consider is the case $k=2$ with $v_i = e_i$ the basis directions. While we believe that the conclusions of Theorem \ref{t:main} hold true for all $k \ge 2$ and $v_i$, we conjecture that the nodal domains $\{f \le 0\}$ have different quantitative behaviour in the case $k \ge 3$ compared to $k=2$, for instance under mild conditions on $g_i$ we believe that the critical exponents match those of Bernoulli percolation as soon as $k \ge 3$ (see \cite[Conjecture 1.3]{pet08} for a related conjecture for corner percolation).

\medskip
\section{Proof of the main results}
\label{s:2}

Recall the decomposition $f(x_1, x_2) = g_1(x_1) + g_2(x_2)$, where $g_i$ are stationary Gaussian processes. The advantage of this decomposition is that, in order to analyse the phase transition for $f$, it will suffice to study the extrema of the processes $g_i$ on large intervals. This avoids the need to invoke sharp threshold criteria (as in, for instance, \cite{rv20,mv20,mrv20} following the classical approach introduced in \cite{kes80}).

\smallskip
We begin by stating auxiliary results on the extrema of stationary Gaussian processes. Let $g$ be a continuous centred stationary Gaussian process on $\mathbb{R}$ with covariance kernel $K$. Throughout we assume that $g$ is almost surely $C^2$-smooth and satisfies $K(x) \log |x| \to 0$ as $|x| \to \infty$. For simplicity we also normalise $g$ so that $\mathbb{E}[g(0)^2] = 1$. 

\smallskip
We first recall the well-known scaling limit of the supremum of $g$ on large intervals. Let $L_T = \sqrt{2 \log T}$.

\begin{proposition}[Scaling limit of the supremum; {\cite[Theorem 8.2.7]{llr83}}] 
\label{p:scaling}
As $T \to \infty$,
\begin{equation}
\label{e:scaling}
\frac{\sup_{t \in [0,T]} g(t) - L_T}{ 1/ L_T}  \Rightarrow \mathcal{G}   + \log(\sqrt{\lambda_2}) - \log( 2 \pi)  
\end{equation}
in law, where $\mathcal{G}$ denotes a standard Gumbel random variable (i.e.\ $\mathbb{P}(\mathcal{G} \le x) =  e^{-e^{-x}}$), and $\lambda_2 = \mathbb{E}[g'(0)^2] = -K''(0) \in (0, \infty)$. 
\end{proposition}

More recently the corresponding sharp concentration bounds have been established:

\begin{proposition}[Concentration of the supremum; see {\cite{tan15}}]
\label{p:concen}
There exist $c_1, c_2 > 0$, such that, for each $T \ge 2$ and $x \ge 0$,
\begin{equation}
\label{e:concen}
\mathbb{P} \Big( \big|\sup_{t \in [0,T]} g(t) - L_T \big|  > x/L_T \Big) \le   c_1 e^{- c_2 x}  .
\end{equation}
\end{proposition} 

This result is essentially due to Tanguy \cite{tan15}, although he proved it under the additional assumption that the covariance kernel $K$ is non-increasing, and hence only for processes which are positively-correlated. Since we wish to work with non-positively-correlated processes, in Section \ref{s:concen} we show how to adapt the proof of \cite{tan15} to lift this assumption. The bound in \eqref{e:concen} is sharp in the sense that it captures the exponential (right-)tail of the limiting Gumbel random variable in \eqref{e:scaling} on the correct scale $1/L_T$; indeed the classical statement of Gaussian concentration of the supremum (see, e.g., \cite[Theorem 2.1]{a90})
\begin{equation}
\label{e:clacon}
\mathbb{P} \Big( \big|\sup_{t \in [0,T]} g(t)  - L_T \big|  > x \Big) \le   2 e^{-  x^2/2 }  
\end{equation}
would not be sufficient for our purposes.

\smallskip
Finally we state a stronger version of Proposition \ref{p:scaling} that gives the full point process convergence for local maxima and minima of $g$:

\begin{proposition}[Point process convergence of local minima and maxima]
\label{p:ppc}
Let $(m_i)_{i \ge 1}$ and $(n_i)_{i \ge 1}$ denote the positions of the local maxima and minima of $g$ respectively (since $g$ is $C^2$-smooth these are countable). Then, as $T \to \infty$,
\[ \Big( \sum_{i \ge 1} \delta_{(m_i/ T,  L_T( g(m_i) - L_T) )} ,  \sum_{i \ge 1} \delta_{(n_i/ T,  L_T(- g(n_i) - L_T) )}  \Big) \Rightarrow ( \mathcal{M}, \mathcal{N}  )   \]
where $\delta_{(x, y)}$ denotes a Dirac mass at $(x, y) \in \mathbb{R}^2$, $\mathcal{M}$ and $\mathcal{N}$ are independent Poisson point processes on $\mathbb{R} \times (-\infty, \infty]$ with intensity 
\[ \frac{\sqrt{\lambda_2}}{2\pi}  dx \otimes e^{-y} dy , \] 
and $\Rightarrow$ denotes vague convergence in the sense of point processes.
\end{proposition}

The point process convergence for local maxima (or, equivalently, local mimima) is stated in \cite[Theorem 9.5.2]{llr83}, and see \cite[Theorem 11.1.5]{llr83} and the discussion thereafter for the independence of the limiting point processes for maxima and minima. The main consequence of Proposition \ref{p:ppc} that we draw is that suprema on disjoint intervals, infima on disjoint intervals, as well as the supremum and infimum on \textit{any} intervals, are all jointly asymptotically independent.

\smallskip
We now show how Propositions \ref{p:scaling}--\ref{p:ppc} imply our main results, beginning with a simple consequence of these propositions:

\begin{proposition}
\label{p:supinf}
For every $h \in \mathbb{R}$,
\begin{align}
\label{e:supinf1}
&  \lim_{T \to \infty}  \mathbb{P} \Big( \sup_{t \in [-2T,-T]} g(t) > L_T, \sup_{t \in [T, 2T]} g(t) > L_T, \inf_{t \in [-2T, 2T]} g(t) > -L_T \Big)  \\
 \nonumber &  \qquad   \ \ =  \mathbb{P} \big( \mathcal{G} \! > \! -\log(\sqrt{\lambda_2}) + \log(2 \pi) \big)^2  \, \mathbb{P}\big( \mathcal{G} \! <\! -\log(\sqrt{\lambda_2}) + \log(2 \pi) - \log 4  \big) .
 \end{align}
Moreover, for every $c > 0$ there exist $c_1, c_2 > 0$ such that, for every $T \ge 2$,
\begin{equation}
\label{e:supinf2}
\mathbb{P} \Big( \sup_{t \in [0,T]} g(t) > L_T - c , \inf_{t \in [0, T]} g(t) > -L_T - c \Big)  \ge 1 - c_1 e^{- c_2  \sqrt{\log T} }  . 
  \end{equation}
\end{proposition}

\begin{proof}
An elementary computation shows that, for any $s > 0$,
\[  \frac{L_T - L_{sT}}{1/L_{sT}} \to -\log s  \]
as $T \to \infty$. Hence by Proposition \ref{p:scaling} and the equality in law of $g$ and $-g$, as $T \to \infty$,
\[ \mathbb{P} \Big( \sup_{t \in [0, sT]} g(t) > L_T  \Big) \to \mathbb{P}\Big( \mathcal{G} > - \log(\sqrt{\lambda_2}) + \log(2 \pi) - \log s  \Big)  . \]
By the point process convergence in Proposition \ref{p:ppc} we have 
\begin{align*}
 & \lim_{T \to \infty} \mathbb{P} \Big( \sup_{t \in [-2T,-T]} g(t) > L_T, \sup_{t \in [T, 2T]} g(t) > L_T, \inf_{t \in [-2T, 2T]} g(t) > -L_T \Big) \\
 & \qquad  = \lim_{T \to \infty} \mathbb{P} \Big( \sup_{t \in [0,T]} g(t) > L_T \Big)^2  \mathbb{P} \Big( \sup_{t \in [0, 4T]} g(t)  <  L_T \Big)   
\end{align*}  
where we also used stationarity and the equality in law of $g$ and $-g$. The first statement then follows from \eqref{e:supinf1}. 

For the second statement we instead use the union bound
\[ \mathbb{P}(A \cap B) \ge 1 - \mathbb{P}(A^c) - \mathbb{P}(B^c)  , \quad \text{for all events } A, B,\]
together with the equality in law of $g$ and $-g$, to deduce that
\begin{align*}
&   \mathbb{P} \Big( \sup_{t \in [0,T]} g(t) > L_T - c , \inf_{t \in [0, T]} g(t) > -L_T - c \Big)   \\
&   \qquad  \ge 1 -  \mathbb{P} \Big( \sup_{t \in [0,T]} g(t) < L_T - c  \Big) - \mathbb{P} \Big( \sup_{t \in [0, T]} g(t) > L_T + c  \Big)  \\
& \qquad = 1 -  \mathbb{P} \Big( \big| \sup_{t \in [0, T]} g(t) - L_T \big| >  c \Big)  . 
\end{align*}
The result then follows from Proposition \ref{p:concen}.
\end{proof}

\begin{proof}[Proof of Theorem \ref{t:main}]
 Recall the decomposition $f(x_1, x_2) =  g_1(x_1) + g_2(x_2)$, where $g_i$ are centred stationary $C^2$-smooth Gaussian processes with covariance $K_i$. By applying a linear rescaling to the domain of $g$, without loss of generality we may assume that $\mathbb{E}[g_1'(0)^2] = \mathbb{E}[g_2'(0)^2] = 1$. Define $L_{1;T} = \sqrt{2 K_1(0) \log T}$ and $L_{2; T} = \sqrt{2 K_2(0) \log T}$, and define also $\tau(T) = T^{K_1(0)/K_2(0)}$ so that $L_{1;T} = L_{2;\tau(T)}$.  

\smallskip
We begin with the first statement. By monotonicity, it is sufficient to show that $\{f \le 0\}$ has bounded connected components almost surely. We say that a rectangle $R = [a_1, b_1] \times [a_2 , b_2]$ is \textit{blocking} if
\[    \min\{ g_1(a_1), g_1(b_1)  \} >  - \inf_{t \in [a_2, b_2] }  g_2(t)  \quad \text{and} \quad \min\{ g_2(a_2), g_2(b_2) \} > - \inf_{t \in [a_1, b_1] }  g_1(t) . \]
The relevance of a \textit{blocking} rectangle $R$ is that $f(x_1,x_2) = g_1(x_1) + g_2(x_2)> 0$ on the boundary~$\partial R$. Now, suppose there exists an $s \in \R$ such that 
\[  \sup_{t \in [-2T,-T]} g_1(t) > s, \sup_{t \in [T, 2T]} g_1(t) > s, \inf_{t \in [-2T, 2T]} g_1(t) > - s \]
and
\[ \sup_{t \in [-2\tau(T),-\tau(T)]} g_2(t) > s, \sup_{t \in [\tau(T), 2\tau(T)]} g_2(t) > s, \inf_{t \in [-2\tau(T), 2\tau(T)]} g_2(t) > - s . \]
Then the rectangle $R = [a_1, b_1] \times [a_2, b_2]$ is blocking, where
\[    a_1 = \argmax \{ g_1(t) : t \in [-2T, -T] \} \ , \quad b_1 = \argmax \{ g_1(t) : t \in [T, 2T] \}  , \]
\[    a_2 = \argmax \{ g_2(t) : t \in [-2\tau(T), -\tau(T)] \}  \  , \quad  b_2 = \argmax \{ g_2(t) : t \in [\tau(T), 2\tau(T)] \}  , \]
breaking ties arbitrarily if necessary. Since also 
\[ R \supset [-T, T] \times [-\tau(T), \tau(T)] ,\]
together with the independence of $g_1$ and $g_2$ we deduce that, for each $T > 0$,
\begin{align*}
& \mathbb{P}( \text{there exists a blocking rectangle containing $[-T, T] \times [-\tau(T), \tau(T)] $} ) \\
&   \quad \ge  \sup_{s \in \mathbb{R}}  \  \mathbb{P} \Big( \sup_{t \in [-2T,-T]} g_1(t) > s, \sup_{t \in [T, 2T]} g_1(t) > s, \inf_{t \in [-2T, 2T]} g_1(t) > - s \Big)  \\
& \quad  \quad \times \mathbb{P} \Big( \sup_{t \in [-2\tau(T),-\tau(T)]} g_2(t) > s, \sup_{t \in [\tau(T), 2\tau(T)]} g_2(t) > s, \inf_{t \in [-2\tau(T), 2\tau(T)]} g_2(t) > - s \Big)  .
\end{align*}
Setting $s = s(T) = L_{1;T} = L_{2;\tau(T)}$ in the above, by \eqref{e:supinf1}
\[\liminf_{T \to \infty} \mathbb{P}( \text{there exists a blocking rectangle containing $[-T, T] \times [-\tau(T), \tau(T)] $} )  > 0 .\]

\smallskip
To finish the proof we use an ergodic argument inspired by the `box lemma' in \cite{gkr88}. Since $g_1$ and $g_2$ are stationary Gaussian processes with correlation decaying at infinity, they are ergodic (Maruyama's theorem). Moreover, being independent, they are in fact jointly ergodic, and so $f$ is ergodic with respect to the `diagonal' shift $\theta_s := (x_1, x_2) \mapsto (x_1 + s, x_2 + s)$. Now let $\mathcal{E}$ be the event that for every rectangle $R_0$ there exists a blocking rectangle containing $R_0$. By monotonicity (recall also that $\tau(T) \to \infty$),
\[ \mathbb{P}(\mathcal{E}) \ge \liminf_{T \to \infty} \mathbb{P}(\text{there is a blocking rectangle containing $[-T, T] \times [-\tau(T), \tau(T)] $}) > 0 .\]
On the other hand, the event $\mathcal{E}$ is invariant under $\theta_s$, and so by ergodicity $\mathbb{P}(\mathcal{E}) = 1$. Since the existence of a blocking rectangle containing $R_0$ implies that $R_0$ is disconnected from infinity in $\{f \le 0\}$, on the event $\mathcal{E}$ the excursion set $\{f \le 0\}$ has only bounded connected components, which completes the proof.

\smallskip
We turn to the second statement, which we prove by adapting the classical construction of the unique infinite cluster in supercritical planar Bernoulli percolation. Fix $\ell > 0$, and for $n \in \mathbb{N}$ define $T_n = 2^n$ and the event
\begin{align*}
 \mathcal{S}_n & := \Big\{     \sup_{t \in [0,2T_n]} g_1(t) > L_{1;2T_n} - \ell/2, \inf_{t \in [0, 2T_n]} g_1(t) > - L_{1;2T_n} -\ell/2        \Big\} \\
 & \qquad \quad \cap \Big\{     \sup_{t \in [0,2 \tau(T_n)]} g_2(t) > L_{1;2T_n} - \ell/2, \inf_{t \in [0, 2 \tau(T_n)]} g_2(t) > - L_{1;2T_n} - \ell/2       \Big\} .
 \end{align*} 
 Note that, on the event $\mathcal{S}_n$, $f(x_1,x_2) = g_1(x_1) + g_2(x_2) >  -\ell$ on the line-segments
 \[    \{a_1\} \times [0, 2\tau(T_n)] \quad \text{and} \quad [0, 2T_n] \times \{a_2\}  ,  \]
 where
 \[ a_1 =      \argmax \{ g_1(t) : t \in [0, T_n] \} \quad \text{and} \quad      a_2 =  \argmax \{ g_2(t) : t \in [0, 2 \tau(T_n) ] \}  ,  \]
 breaking ties arbitrarily if necessary. Hence $\mathcal{S}_n$ implies the existence of a top-bottom path in $\{f \ge - \ell\} \cap ([0, T_n] \times [0, 2\tau(T_n)])$ (i.e.\ one that intersects $[0, T_n] \times  \{0\}$ and $[0, T_n] \times \{2 \tau(T_n)\}$) and also the existence of left-right path in $\{f \ge -\ell\} \cap ( [0, 2T_n] \times [0, 2\tau(T_n)])$ (i.e.\ one that intersects $\{0\} \times [0, 2\tau(T_n)]$ and $\{2T_n\} \times [0, 2\tau(T_n)]$). By independence and \eqref{e:supinf2} (recall also that $L_{1;T} = L_{2;\tau(T)}$), there are $c_1,c_2 > 0$ such that for all $n \ge 1$
 \[      \mathbb{P}( \mathcal{S}_n ) \ge 1 -  c_1e^{-c_2 \sqrt{n}} ,   \]
and so by the Borel-Cantelli lemma almost surely there is an $n_0 \ge 1$ such that $\cap_{n \ge n_0} \mathcal{S}_n$ occurs. Recalling that $T_n = 2^n$, this event implies the existence of an unbounded path in $\{f \ge -\ell\}$ that intersects $[0, T_{n_0}] \times \{0\}$, so we have proven that $\{f \ge -\ell\}$ contains an unbounded component almost surely. For uniqueness, recall the event $\mathcal{E}$ from the proof of the first statement of the theorem. This event implies that every compact domain is surrounded by a circuit in $\{f > 0\}$, which precludes the existence of multiple unbounded components in $\{f \ge -\ell\}$. Since $\mathcal{E}$ occurs almost surely, the unbounded component of $\{f \ge -\ell\}$ is therefore unique. Since $f = -f$ in law, this gives the result.
\end{proof}

\begin{proof}[Proof of Theorem \ref{t:3d}]
By monotonicity it suffices to prove the result for $\ell < 0$. Recall from the proof of Theorem \ref{t:main} that the $g_i$ are jointly ergodic. Hence almost surely one can find a $(s_3, \ldots , s_d) \in \mathbb{R}^{d-2}$ such that $g_i(s_i) < 2\ell/(d-2)$ for each $i \in 3, \ldots , d$. Set $\Sigma = \sum_{i = 3, \ldots , d} g_i(s_i) < 2\ell$ and consider the plane $P = \{  (x_1, \ldots , x_d) : x_i = s_i \text{ for } i = 3, \ldots , d\}$. Since $\tilde{f} = (f - \Sigma)|_P$ has the law of a centred planar additive Gaussian field satisfying the assumptions of Theorem~\ref{t:main}, by the second assertion of this theorem $ \{f \le \ell\}|_{P} = \{\tilde{f} \le \ell+\Sigma\} \supset \{ \tilde{f} \le -\ell\}$ contains an unbounded component almost surely.
\end{proof}
\begin{remark}
\label{r:unique}
Note that although the proof of Theorem \ref{t:3d} shows that the unbounded component of $\{f \le \ell\}|_P$ is unique, this does \textit{not} imply that $\{f \le \ell\}$ has a unique unbounded component.
\end{remark}

The proof of Theorems \ref{t:rsw} and \ref{t:cw} are similar to the proof of Theorem \ref{t:main} but we include the details for completeness. Let $\lambda_{2;i} = \mathbb{E}[g_i'(0)^2] = -K_i''(0) \in (0, \infty)$ for $i=1,2$. 

\begin{proof}[Proof of Theorem \ref{t:rsw}]
By applying a linear rescaling to the domain of $f$ (recall that we do \textit{not} assume that $\lambda_{2;i}   = 1$ so this rescaling is without loss of generality) it is enough to prove the result for the square-crossing event $\textrm{Cross}_0(T, T)$.

\smallskip
 Let us consider first the case that $K_1(0) = K_2(0)$ and without loss of generality suppose $K_1(0) = 1$. Recall that $L_T =  \sqrt{2 \log T}$ and define the event
  \begin{align*}\mathcal{A}_T & := \Big\{     \sup_{t \in [0,T]} g_1(t) > L_T  , \inf_{t \in [0, T]} g_1(t) > - L_T     \Big\} \\
 & \qquad \quad \cap \Big\{     \sup_{t \in [0,T]} g_2(t) > L_T , \inf_{t \in [0, T]} g_2(t) > - L_T   \Big\} .
 \end{align*} 
  Note that, on the event $\mathcal{A}_T$, $f(x_1,x_2) = g_1(x_1) + g_2(x_2) > 0$ on the line-segments
 \[    \{a_1\} \times [0, T] \quad \text{and} \quad [0, T] \times \{a_2\}  ,  \]
 where
 \[ a_1 =      \argmax \{ g_1(t) : t \in [0, T] \} \quad \text{and} \quad     a_2 =   \argmax \{ g_2(t) : t \in [0, T ] \}  ,  \]
 breaking ties arbitrarily if necessary. Hence $\mathcal{A}_T$ implies the existence of left-right and top-down paths in $\{f > 0\} \cap [0, T]^2$, and hence precludes $\textrm{Cross}_0(T, T)$ (i.e.\ a left-right path in $\{f \le 0\} \cap [0, T]^2$).  By Propositions \ref{p:scaling} and \ref{p:ppc}, and the symmetry of $g_i$ and $-g_i$ in law, as $T \to \infty$,
\[ \mathbb{P}( \mathcal{A}_T ) \to \prod_{i=1,2} \mathbb{P}( \mathcal{G} > - \log(\sqrt{\lambda_{2;i}}) + \log(2\pi)  )\mathbb{P}( \mathcal{G} < - \log(\sqrt{\lambda_{2;i}}) + \log(2\pi)  )  \in (0, 1) . \]
Since $ \mathcal{A}_T$ and $\textrm{Cross}_0(T, T) $ are disjoint we have
\[ \limsup_{T \to \infty} \mathbb{P}( \textrm{Cross}_0(T, T) )\le 1 -   \lim_{T \to \infty} \mathbb{P}( \mathcal{A}_T ) < 1. \]
By the equality in law of $f$ and $-f$ we also have
\[ \liminf_{T \to \infty} \mathbb{P}( \textrm{Cross}_0(T, T) )  \ge \lim_{T \to \infty} \mathbb{P}( \mathcal{A}_T )  > 0  , \] 
which proves the first statement.
 
 \smallskip
Now suppose that $K_1(0) < K_2(0)$ (the case $K_2(0) < K_1(0)$ is almost identical and we omit it). Recall that $L_{i;T} = \sqrt{2 K_i(0) \log T}$ for $i=1,2$, and define the event
\[ \mathcal{B}_T^h  = \Big\{     \sup_{t \in [0,T]} g_2(t) > L_{2;T} - h/\sqrt{\log T}   , \inf_{t \in [0, T]} g_1(t) > - L_{1;T}  - h/\sqrt{\log T}    \Big\}  . \]
Note that, on the event $\mathcal{B}_T^h$, $f(x_1,x_2) = g_1(x_1) + g_2(x_2) > L_{2;T} - L_{1;T} - 2h/\sqrt{\log T} $ on the line-segment $[0, T] \times \{a\}$, where $a =    \argmax \{ g_2(t) : t \in [0, T ] \}$, breaking ties arbitrarily if necessary. Hence $\mathcal{B}_T^h$ implies the existence of a left-right path in $\{f > L_{2;T} - L_{1;T} - 2h/\sqrt{\log T}\} \cap [0, T]^2$. Moreover, if $h \in \mathbb{R}$ is fixed and $T$ taken sufficiently large so that
 \[  L_{2;T} - L_{1;T} > 2h/\sqrt{\log T} ,  \]
 then $\mathcal{B}_T^h$ implies also the existence of a left-right path in $\{f \ge 0\} \cap [0, T]^2$, which is an event of equal probability to $\textrm{Cross}_0(T, T) $ by the equality in law of $g_i$ and $-g_i$. Combining with Proposition \ref{p:scaling} we have that
 \begin{align*}
 & \liminf_{T \to \infty} \mathbb{P}( \textrm{Cross}_0(T, T)  ) \ge \lim_{T \to \infty} \mathbb{P}(\mathcal{B}_T^h ) \\
 & \qquad  =  \mathbb{P} \big( \mathcal{G} > -\log(\sqrt{\lambda_{2;2}}) + \log(2 \pi) - \sqrt{2}h \big)  \mathbb{P} \big( \mathcal{G}<  -\log(\sqrt{\lambda_{2;1}}) + \log(2 \pi) + \sqrt{2}h \big)  . 
 \end{align*}
Since the right-hand side of the above tends to $1$ as $h \to \infty$, we deduce the result by taking $h \to \infty$.
\end{proof}

\begin{proof}[Proof of Theorem \ref{t:cw}]
Recall the event
\[ \mathcal{B}_T^h  = \Big\{     \sup_{t \in [0,T]} g_2(t) > L_T - h/\sqrt{\log T}   , \inf_{t \in [0, T]} g_1(t) > - L_T  - h/\sqrt{\log T}    \Big\}  \]
which, as in the proof of the second statement of Theorem \ref{t:rsw}, implies the existence of a left-right path in $\{f > -2h/\sqrt{\log T} \} \cap [0, T]^2$. Hence for fixed $h > 0$ we have 
 \begin{align*}
 & \liminf_{T \to \infty} \mathbb{P}( \textrm{Cross}_{2h/\sqrt{\log T}} (T, T)  )   \ge  \lim_{T \to \infty} \mathbb{P}(\mathcal{B}_T^h )  \\
 &  \qquad    =   \mathbb{P} \big( \mathcal{G} > -\log(\sqrt{\lambda_{2;2}}) + \log(2 \pi) - \sqrt{2}h \big)  \mathbb{P} \big( \mathcal{G} < -\log(\sqrt{\lambda_{2;1}}) + \log(2 \pi) + \sqrt{2}h \big) . 
 \end{align*}
Since the right-hand side of the above tends to $1$ as $h \to \infty$, we deduce the first claim of the theorem. On the other hand, if we define instead
\[ \mathcal{C}_T^h  = \Big\{     \sup_{t \in [0,T]} g_1(t) > L_T + h/\sqrt{\log T}   , \inf_{t \in [0, T]} g_2(t) > - L_T  + h/\sqrt{\log T}    \Big\}  , \]
then, by the same argument, the event $\mathcal{C}_T^h$ implies the existence of a top-bottom path in $\{f >  2h/\sqrt{\log T}\} \cap [0, T]^2$, and hence precludes $\textrm{Cross}_{2h/\sqrt{\log T}} (T, T) $ (i.e.\ a left-right path in $\{f \le 2h/\sqrt{\log T} \} \cap [0, T]^2$). Thus we also have
 \begin{align*}
 &  \limsup_{T \to \infty} \mathbb{P}( \textrm{Cross}_{2h/\sqrt{\log T}} (T, T)  )  \le 1 -  \lim_{T \to \infty} \mathbb{P}(\mathcal{C}_T^h )     \\ &  \quad    =   1 - \mathbb{P} \big( \mathcal{G} > -\log(\sqrt{\lambda_{2;1}}) + \log(2 \pi) + \sqrt{2}h \big)  \mathbb{P} \big( \mathcal{G} <  -\log(\sqrt{\lambda_{2;2}}) + \log(2 \pi) - \sqrt{2}h \big)   .
 \end{align*}
 Observing that the latter quantity is strictly less than one for all $h \in \mathbb{R}$, since $\ell \mapsto \textrm{Cross}_\ell(T, T)$ is increasing this implies the second claim of the theorem.
\end{proof}

\medskip
\section{Concentration of the supremum}
\label{s:concen}

In this section we prove Proposition \ref{p:concen} following closely the approach of \cite{tan15}. Recall that $g$ is a $C^2$-smooth centred stationary Gaussian process with covariance kernel $K$ satisfying $K(0) = 1$ and $K(x) \log |x| \to 0$ as $|x| \to \infty$. As observed in \cite{tan15}, to obtain the required concentration of $\sup_{t \in [0, T]} g(t)$ it is sufficient to prove the following:

\begin{proposition}
\label{p:concen2}
There exists a $c > 0$ such that, for every $T \ge 2$, $\varepsilon \in (0,1)$ and $\theta \in \mathbb{R}$,
\begin{equation}
\label{e:concen2}
  \textrm{Var}[e^{\theta S_{T,\varepsilon}}] \le  c \theta^2 / (\log T)  \times  \mathbb{E}[e^{2\theta S_{T,\varepsilon}}] , 
  \end{equation}
where $S_{T,\varepsilon} = \sup_{t \in [0, T] \cap (\varepsilon \mathbb{Z})} g(t)$. 
\end{proposition}

Before giving the proof, let us explain how it implies Proposition \ref{p:concen}:

\begin{proof}[Proof of Proposition \ref{p:concen}]
Since $S_{T,\varepsilon}$ is non-decreasing as $\varepsilon \to 0$ and converges almost surely to $S_T = \sup_{t \in [0, T] } g(t)$ (recall that $g$ is continuous), by monotone convergence we have
\begin{equation}
\label{e:concen3}
  \textrm{Var}[e^{\theta S_T}] \le  c \theta^2 /(\log T) \times   \mathbb{E}[e^{2\theta S_T}] , 
  \end{equation}
for every $T \ge 2$ and $\theta \in \mathbb{R}$. By a concentration result valid for arbitrary random variables (see \cite[Lemma 6]{tan15}), \eqref{e:concen3} implies the concentration bound
\begin{equation}
\label{e:concen4}
\mathbb{P} \Big( | S_T - \mathbb{E}[S_T] |  > x  \Big) \le   c_1 e^{- c_2 x \sqrt{\log T} }  .
\end{equation}
for some $c_1, c_2 \ge 0$ and every $T \ge 2$ and $x \ge 0$. Finally observe that \eqref{e:scaling} and \eqref{e:concen4} imply that
\[  |\mathbb{E}[S_T]- \sqrt{2 \log T}| = O(1/\sqrt{\log T})  \]
and so, up to adjusting constants, we can replace $\mathbb{E}[S_T]$ in \eqref{e:concen4} with $\sqrt{2 \log T}$.
\end{proof}

In order to prove Proposition \ref{p:concen2} we use the hypercontractivity argument of \cite{tan15} (itself based on an argument of Chatterjee \cite{cha14}), except that we modify some details to allow us to lift the assumption that $K(x)$ is non-increasing. This is similar to arguments that appeared in \cite{mrv20} in a related setting.

\begin{proof}[Proof of Proposition \ref{p:concen2}]
To ease notation let us fix $T \ge 2$ and $\varepsilon \in (0, 1)$ and abbreviate $S = S_{T; \varepsilon}$. In the proof we will also identify $g$ with its restriction to $[0, T] \cap (\varepsilon \mathbb{Z})$. Let $\tilde{g}$ be an independent copy of $g$, and for each $s \ge 0$ let
\begin{equation}
\label{e:inter}
g_s(\cdot) = e^{-s} g(\cdot) + \sqrt{1 - e^{-2s} } \tilde{g}(\cdot) ; 
\end{equation}
this defines an interpolation from $g$ to $\tilde{g}$ along the Ornstein-Uhlenbeck semigroup \cite{cha14}. Then via a classical interpolation argument for the variance of functions of centred Gaussian vectors (see \cite{tan15,cha14}), one has the exact formula
\begin{equation}
\label{e:var}
    \textrm{Var}[e^{\theta S}] =  \theta^2 \sum_{ x,y \in   [0, T] \cap (\varepsilon \mathbb{Z})} K(x-y) \int_0^\infty  e^{-s}  \mathbb{E}\Big[  e^{\theta S }e^{\theta S_s  } \id_{I = x, I_s = y}  \Big]  ds  
    \end{equation}
where $I = \{x \in [0, T] \cap (\varepsilon \mathbb{Z}) : g(x) = S\}$ denotes the index of the maximum of $g$ (recall that we identify $g$ with its restriction to $[0, T] \cap (\varepsilon \mathbb{Z})$), and $S_s$ and $I_s$ are defined analogously to $S$ and $I$ with $g_s$ replacing $g$.

Now let $\gamma \in (0, 1)$ be a constant (not depending on $T, \varepsilon$ or $\theta$) whose value we will fix later. For $i \in \{0, 1,\ldots , \lceil T^{1-\gamma}  \rceil\}$ define $B_i = [i T^\gamma, (i+1) T^\gamma) \cap ([0, T] \cap (\varepsilon \mathbb{Z}) )$. Then we can rewrite \eqref{e:var} as
\begin{align}
\label{e:var2}
&    \theta^2 \sum_{i, j : |i-j| \le 1} \sum_{ x \in B_i, y \in B_j}  K(x-y) \int_0^\infty  e^{-s}  \mathbb{E}\Big[  e^{\theta S }e^{\theta S_s  } \id_{I = x, I_s = y}  \Big]  ds  \\
\nonumber   & \qquad +  \theta^2  \sum_{i , j : |i-j| \ge 2}  \sum_{ x \in B_i, y \in B_j} K(x-y) \int_0^\infty  e^{-s}  \mathbb{E}\Big[  e^{\theta S }e^{\theta S_s  } \id_{I = x, I_s = y}  \Big]  ds   .
\end{align}
To deal with the second term in \eqref{e:var2} we simply observe that if $|i-j| \ge 2$, $x \in B_i$ and $y \in B_j$, then $|x-y| \ge  T^\gamma$ and so $K(x-y) \le c_1 /\log T$ for some $c_1 > 0$ (which depends on $\gamma$).  Along with the Cauchy-Schwarz inequality and the equality in law of $g$ and $g_s$, we deduce that 
\begin{equation}
\label{e:var3}  \textrm{second term in \eqref{e:var2}} \le c_1  \theta^2 / (\log T)  \times \sup_{s \ge 0}  \mathbb{E}\big[  e^{\theta S} e^{\theta S_s}    \big] \le c_1  \theta^2 / (\log T) \times \mathbb{E}\big[  e^{2\theta S} ]   .
\end{equation}
To handle the first term in \eqref{e:var2} we instead use $K(x-y) \le K(0) = 1$ to arrive at the bound
\[  \textrm{first term in \eqref{e:var2}} \le \theta^2 \sum_{i}   \int_0^\infty  e^{-s}  \mathbb{E}\big[  e^{\theta S }e^{\theta S_s  } \id_{I, I_s \in B_i^+  }  \big]  ds   \]
where $ B_i^+  = \cup_{|j-i| \le 1} B_j$. We next exploit the hypercontractivity of the Ornstein-Uhlenbeck semigroup (see \cite{cha14}): for any function $F: g \to \mathbb{R}$ and any $p, q \ge 1$ and $s\ge 0$ such that $e^{2s} \ge (q-1)/(p-1)$,
\[ \mathbb{E}[ | \mathbb{E}[ F(g_s) | \mathcal{F}  ] |^q  ]^{1/q}  \le   \mathbb{E}[ |F(g)|^p ]^{1/p} , \]
where $\mathcal{F}$ denotes the $\sigma$-algebra generated by $g$. Define $p(s) = 1 + e^{-s} \in (1,2]$ and its H\"{o}lder complement $q(s) = 1+e^s \in [2, \infty)$, and note that $e^{2s} = (q(s) -1)/(p(s)-1)$ and $(2-p(s))/p(s) = \textrm{tahn}(s/2)$. Recalling the equality in law of $g$ and $g_s$, applying first H\"{o}lder's inequality (with $p' = p(s)$ and $q' = q(s)$), then hypercontractivity (to the function $F = e^{\theta S } \id_{I  \in B_i^+}$, still  with $p' = p(s)$ and $q' = q(s)$), then again H\"{o}lder's inequality (with $p' = 2/p(s)$ and $q' = 2/(2-p(s))$), we have
\begin{align*}
 \mathbb{E} \big[  e^{\theta S }e^{\theta S_s  } \id_{I, I_s \in B_i^+  }  \big]  &   \le  \mathbb{E} \big[ \big(  e^{\theta S } \id_{I  \in B_i^+  } \big)^{p(s)} \big]^{1/p(s)}  \times  \mathbb{E} \big[ \mathbb{E}[ e^{\theta S_s } \id_{I_s  \in B_i^+  } | \mathcal{F} ]^{q(s)} \big]^{1/q(s)}  \\
 & \le  \mathbb{E} \big[ \big(  e^{\theta S } \id_{I  \in B_i^+  } \big)^{p(s)} \big]^{2/p(s)}  \\
 & \le   \mathbb{P}(I  \in B_i^+ )^{\textrm{tahn}(s/2)}   \mathbb{E} \big[  e^{2\theta S } \id_{I  \in B_i^+  }    \big]   . 
 \end{align*}
Since one can check that
\[ \int_0^\infty e^{-s} \alpha^{\textrm{tahn}(s/2) }  ds \le 2/|\log \alpha| \]
for every $\alpha \in (0,1)$, by the union bound $  \mathbb{P}(I  \in B_i^+ ) \le   |B_i^+| \sup_{ y \in \mathbb{R}}     \mathbb{P}(I  \in [y,y+1] )  $ we see that
\begin{equation}
\label{e:var4}
  \textrm{first term in \eqref{e:var2}} \le  2   \theta^2  \times \frac{1}{ | \log (3 T^\gamma) +  \sup_{ y \in \mathbb{R}}     \log \mathbb{P}(I  \in [y,y+1] )   |  } \times \mathbb{E}\big[  e^{2\theta S} ]  
 \end{equation}
 as long as $\log (3 T^\gamma) -  \sup_{ y \in \mathbb{R}}   \log  \mathbb{P}(I  \in [y,y+1] )   < 0$.

It remains to establish the `polynomial delocalisation' of the maximiser
\begin{equation}
\label{e:var5}
 \sup_{y \in \mathbb{R}} \mathbb{P}(I \in [y,y+1] )   \le c_2 T^{-c_3} 
 \end{equation}
 for some $c_2, c_3 > 0$ and all $T \ge 2$ and $\varepsilon \in (0, 1)$, since then we can choose $\gamma \in (0, c_3)$ and combine \eqref{e:var3} and \eqref{e:var4} to deduce the result. Before proving \eqref{e:var5}, remark that heuristically it should follow from the fact that $S$ is likely to be of order $\sqrt{\log T}$ whereas $\sup_{x \in [y,y+1]} g(x)$ is of unit order. So let us first establish the lower bound 
 \begin{equation}
 \label{e:var6}
  \mathbb{E}[S] \ge c_4 \sqrt{\log T} 
  \end{equation}
  for some $c_4 > 0$ and all $T \ge 2$ and $\varepsilon \in (0,1)$. By taking $T$ sufficiently large and extracting a thinned subset of the indices in $[0, T] \cap(\varepsilon \mathbb{R})$ we see that $S$ dominates the maximum of a centred normalised Gaussian vector $X = (X_i)_{1 \le i \le n}$ where $n \ge c_5 T$, $\mathbb{E}[X_i^2]=1$ and $\mathbb{E}[(X_i  - X_j)^2] \ge \delta$ for some $c_5, \delta > 0$ and every $i \neq j$. Hence by the Sudakov-Fernique inequaity \cite[Theorem 2.9]{a90}
  \[  \mathbb{E}[S] \ge \mathbb{E}[\max_i X_i] \ge  \sqrt{\delta} \, \mathbb{E}[  \max_i Y_i ] \]
  where $Y = (Y_i)_{1 \le i \le n}$ are i.i.d.\ standard Gaussians, and \eqref{e:var6} follows. Now let $m = \mathbb{E}[\sup_{x \in [0, 1]} g(x) ] \in (0, \infty)$. Then for any $y \in \mathbb{R}$, $\mathbb{P}(I \in [y,y+1])$ is bounded above by
 \begin{align*}
 \mathbb{P} \big( \sup_{x \in [y, y+ 1]} g(x) \ge S \big) & \le \mathbb{P} \big( \sup_{x \in [0,1] } g(x) \ge   \mathbb{E}[S]/2  \big) +  \mathbb{P}( S \le  \mathbb{E}[S]/2 ) \\
  &  \le  2 e^{- ( \mathbb{E}[S]/2 - m)^2/2} + 2 e^{- \mathbb{E}[S]^2 / 8}    \le c_2 T^{-c_3} 
  \end{align*}
 where we used classical concentration of the supremum \eqref{e:clacon} and the lower bound \eqref{e:var6}. Hence we have established \eqref{e:var5}, which completes the proof.
\end{proof}


\bigskip

\bibliographystyle{plain}
\bibliography{biblio}

\end{document}